\documentclass[a4paper]{amsart}

\usepackage{amssymb}

\newtheorem{Th}{Theorem}[section]

\newtheorem{Lemma}[Th]{Lemma}

\theoremstyle{definition}
\newtheorem{Example}[Th]{Example}
\newtheorem{df}[Th]{Definition}

\theoremstyle{remark}
\newtheorem{Remark}[Th]{Remark}

\newcommand{\R}{\mathbb{R}}

\begin{document}

\title[Local Pareto optimality on time scales]{%
Necessary and sufficient conditions for local Pareto optimality on time scales}



\author[A. B. Malinowska\ and\ D. F. M. Torres]{%
Agnieszka B. Malinowska \and Delfim F. M. Torres}

\address{Faculty of Computer Science,
Bia{\l}ystok Technical University,
15-351 Bia\l ystok, Poland}

\email{abmalina@pb.bialystok.pl}

\address{Department of Mathematics,
University of Aveiro,
3810-193 Aveiro, Portugal}

\email{delfim@ua.pt}

\thanks{Research partially supported by the
{\it Centre for Research on Optimization and Control} (CEOC) from
the {\it Portuguese Foundation for Science and Technology} (FCT),
cofinanced by the European Community Fund FEDER/POCI 2010 and by KBN
under Bia{\l}ystok Technical University Grant S/WI/1/07.}

\maketitle


\begin{abstract}We study a multiobjective variational problem on time scales.
For this problem, necessary and sufficient conditions for weak
local Pareto optimality are given. We also prove a necessary
optimality condition for the isoperimetric problem
with multiple constraints on time scales.
\end{abstract}


\section{Introduction}

The calculus on time scales was initiated by Aulbach and Hilger
(see \textrm{e.g.} \cite{b2}) in order to create a theory that can
unify discrete and continuous analysis. Since then, much active
research has been observed all over the world (see \textrm{e.g.}
\cite{b1,Zbig, b7, b4} and references therein). In this paper we
consider multiobjective variational problems on time scales
(Section~\ref{subsec:po}). By developing a theory for
multiobjective optimization problems on a time scale, one obtains
more general results that can be applied to discrete, continuous
or hybrid domains. To the best of the authors' knowledge, no study
has been done in this field for time scales. The main results of
the paper provide methods for identifying weak locally Pareto
optimal solutions; versions for continuous domain one can find
\textrm{e.g.} in \cite{b3, b6, b5}. We show that necessary
optimality conditions for isoperimetric problems are also
necessary for local Pareto optimality for a multiobjective
variational problem on a time scale (Theorem~\ref{nc}), and the sufficient
condition for local Pareto optimality can be reduced to the
sufficient optimal condition for a basic problem of the calculus
of variations on a time scale (Theorem~\ref{sc}). We also prove a necessary
optimality condition for the isoperimetric problem with multiple
constraints on time scales (Section~\ref{subsec:IP}).


\section{Time scales calculus}

In this section we introduce basic definitions and results that will
be needed for the rest of the paper. For a more general theory of
calculus on time scales, we refer the reader to \cite{livro}.

A nonempty closed subset of $\mathbb{R}$ is called a \emph{time
scale} and it is denoted by $\mathbb{T}$.

The \emph{forward jump operator}
$\sigma:\mathbb{T}\rightarrow\mathbb{T}$ is defined by
$$\sigma(t)=\inf{\{s\in\mathbb{T}:s>t}\},
\mbox{ for all $t\in\mathbb{T}$},$$
while the \emph{backward jump operator}
$\rho:\mathbb{T}\rightarrow\mathbb{T}$ is defined by
$$\rho(t)=\sup{\{s\in\mathbb{T}:s<t}\},\mbox{ for all
$t\in\mathbb{T}$},$$ with $\inf\emptyset=\sup\mathbb{T}$
(\textrm{i.e.} $\sigma(M)=M$ if $\mathbb{T}$ has a maximum $M$)
and $\sup\emptyset=\inf\mathbb{T}$ (\textrm{i.e.} $\rho(m)=m$ if
$\mathbb{T}$ has a minimum $m$).

A point $t\in\mathbb{T}$ is called \emph{right-dense},
\emph{right-scattered}, \emph{left-dense} and
\emph{left-scattered} if $\sigma(t)=t$, $\sigma(t)>t$, $\rho(t)=t$
and $\rho(t)<t$, respectively.

Throughout the paper we let $\mathbb{T}=[a,b]\cap\mathbb{T}_{0}$
with $a<b$ and $\mathbb{T}_0$ a time scale
containing $a$ and $b$.

\begin{Remark}
The time scales $\mathbb{T}$ considered in this work have a
maximum $b$ and, by definition, $\sigma(b) = b$.
\end{Remark}

The \emph{graininess function}
$\mu:\mathbb{T}\rightarrow[0,\infty)$ is defined by
$$\mu(t)=\sigma(t)-t,\mbox{ for all $t\in\mathbb{T}$}.$$

Following \cite{livro}, we define
$\mathbb{T}^k=\mathbb{T}\backslash(\rho(b),b]$,
$\mathbb{T}^{k^2}=\left(\mathbb{T}^k\right)^k$.

We say that a function $f:\mathbb{T}\rightarrow\mathbb{R}$ is
\emph{delta differentiable} at $t\in\mathbb{T}^k$ if there exists
a number $f^{\Delta}(t)$ such that for all $\varepsilon>0$ there
is a neighborhood $U$ of $t$ (\textrm{i.e.}
$U=(t-\delta,t+\delta)\cap\mathbb{T}$ for some $\delta>0$) such
that
$$|f(\sigma(t))-f(s)-f^{\Delta}(t)(\sigma(t)-s)|
\leq\varepsilon|\sigma(t)-s|,\mbox{ for all $s\in U$}.$$ We call
$f^{\Delta}(t)$ the \emph{delta derivative} of $f$ at $t$ and say
that $f$ is \emph{delta differentiable} on $\mathbb{T}^k$ provided
$f^{\Delta}(t)$ exists for all $t\in\mathbb{T}^k$.

For delta differentiable functions $f$ and $g$, the next formula
holds:

\begin{equation*}
\begin{aligned}
(fg)^\Delta(t)&=f^\Delta(t)g^\sigma(t)+f(t)g^\Delta(t)\\
&=f^\Delta(t)g(t)+f^\sigma(t)g^\Delta(t),
\end{aligned}
\end{equation*}
where we abbreviate here and throughout the text $f\circ\sigma$ by
$f^\sigma$.

A function $f:\mathbb{T}\rightarrow\mathbb{R}$ is called
\emph{rd-continuous} if it is continuous at right-dense points and
if its left-sided limit exists at left-dense points. We denote the
set of all rd-continuous functions by C$_{\textrm{rd}}$ and the set
of all delta differentiable functions with rd-continuous derivative
by C$_{\textrm{rd}}^1$.

It is known that rd-continuous functions possess an
\emph{antiderivative}, \textrm{i.e.} there exists a function $F$
with $F^\Delta=f$, and in this case the \emph{delta integral} is
defined by $\int_{c}^{d}f(t)\Delta t=F(c)-F(d)$ for all
$c,d\in\mathbb{T}$. The delta integral has the following property:
\begin{equation*}
\int_t^{\sigma(t)}f(\tau)\Delta\tau=\mu(t)f(t).
\end{equation*}

We now present the integration by parts formulas for the delta
integral:

\begin{Lemma}(\cite{livro})
\label{integracao partes} If $c,d\in\mathbb{T}$ and
$f,g\in$C$_{\textrm{rd}}^1$, then

\begin{equation*}
\int_{c}^{d}f(\sigma(t))g^{\Delta}(t)\Delta t
=\left[(fg)(t)\right]_{t=c}^{t=d}-\int_{c}^{d}f^{\Delta}(t)g(t)\Delta
t;
\end{equation*}

\begin{equation*}
\int_{c}^{d}f(t)g^{\Delta}(t)\Delta t
=\left[(fg)(t)\right]_{t=c}^{t=d}-\int_{c}^{d}
f^{\Delta}(t)g(\sigma(t))\Delta t.
\end{equation*}

\end{Lemma}

We say that $f:\mathbb{T}\rightarrow\mathbb{R}^{n}$ is a
\emph{rd-continuous} (a \emph{delta differentiable}) function if
each component of $f$, $f_{i}:\mathbb{T}\rightarrow\mathbb{R}$, is a
rd-continuous (a delta differentiable) function. By abuse of
notation, we continue to write C$_{\textrm{rd}}$ for the set of all
rd-continuous vector valued functions and C$_{\textrm{rd}}^1$ for
the set of all delta differentiable vector valued functions with
rd-continuous derivative.

The following Dubois-Reymond lemma for the calculus of variations on
time scales will be useful for our purposes.

\begin{Lemma}(Lemma of Dubois-Reymond \cite{b7})
\label{lemma:DR}
Let $g\in C_{\textrm{rd}}$,
$g:[a,b]^k\rightarrow\mathbb{R}^n$. Then,
$$\int_{a}^{b}g(t) \cdot \eta^\Delta(t)\Delta t=0  \quad
\mbox{for all $\eta\in C_{\textrm{rd}}^1$ with
$\eta(a)=\eta(b)=0$}$$ if and only if $g(t)=c \mbox{ on $[a,b]^k$
for some $c\in\mathbb{R}^n$}$.
\end{Lemma}


\section{Main Results}

We begin by proving necessary optimality conditions for
isoperimetric problems on time scales (\S\ref{subsec:IP}). In
\S\ref{subsec:po} we show that Pareto solutions of multiobjective
variational problems on time scales are minimizers of a certain
family of isoperimetric problems on time scales.

\subsection{Isoperimetric problem on time scales}
\label{subsec:IP}

\begin{df}
For $f:[a,b]\rightarrow \mathbb{R}^{n}$ we define the norm
\begin{equation*}
    \|f\|_{C^{1}_{rd}}=
    \max_{t\in[a,b]^{k}}\|f^{\sigma}(t)\|
    +\max_{t\in[a,b]^{k}}\|f^{\triangle}(t)\| \, ,
\end{equation*}
where $\|\cdot\|$ stands for any norm in $\mathbb{R}^n$.
\end{df}

Let $\mathcal{L}:C^{1}_{rd}\rightarrow \R $ be a functional
defined on the function space $C^{1}_{rd}$ endowed with the norm
$\| \cdot \|_{C^{1}_{rd}}$ and let $A\subseteq C^{1}_{rd}$.
\begin{df}
A function $\hat{f}\in A$ is called a \emph{weak local minimum} of
$\mathcal{L}$ provided there exists $\delta >0$ such that
$\mathcal{L}[\hat{f}]\leq \mathcal{L}[f]$ for all $f\in A$ with
$\|f-\hat{f}\|_{C^{1}_{rd}}<\delta$.
\end{df}

Now, let us consider a functional of the form
\begin{equation}\label{vp}
     \mathcal{L}[y]=\int_{a}^{b}L(t,y^{\sigma}(t),y^{\triangle}(t))\triangle
     t ,
\end{equation}
where $a, b\in \mathbb {T}$ with $a<b$, $L(t,s,v):[a,b]^{k}\times
\R^{n} \times \R^{n} \rightarrow \R$ has partial continuous
derivatives with respect to the second and third variables for all
$t\in[a,b]^{k}$, and $L(t,\cdot,\cdot)$ and its partial derivatives
are rd-continuous at $t$. The \emph{isoperimetric problem} consists
of finding a function
$y$ satisfying:\\
(i) the boundary conditions
\begin{equation}\label{bc}
    y(a)=\alpha\, , \quad y(b)=\beta \, , \quad \alpha, \beta\in \R^{n} \, ;
\end{equation}
and \\
(ii) constraints of the form
\begin{equation}\label{con}
    \mathcal{G}_{i}[y]=\int_{a}^{b}G_{i}(t,y^{\sigma}(t),y^{\triangle}(t))\triangle
    t=\xi_{i},  \quad i=1,\ldots m ,
\end{equation}
where $\xi_{i},i=1,\ldots m$, are specified real constrains,
$G_{i}(t,s,v):[a,b]^{k}\times \R^{n} \times \R^{n} \rightarrow \R$,
$i= 1,\ldots m $, have partial continuous derivatives with respect
to the second and third variables for all $t\in[a,b]^{k}$, and
$G_{i}(t,\cdot,\cdot)$ and their partial derivatives are
rd-continuous at $t$; that takes \eqref{vp} to a minimum.

\begin{df}
Let $\mathcal{L}$ be a functional defined on $C^{1}_{rd}$. The first
variation of $\mathcal{L}$ at $y \in C^{1}_{rd}$ in the direction
$\eta\in C^{1}_{rd}$, also called \emph{G\^{a}teaux derivative} with
respect to $\eta$ at $y$, is defined as
\begin{equation*}
    \delta\mathcal{L}[y;\eta]=lim_{\varepsilon\rightarrow
    0}\frac{\mathcal{L}[y+\varepsilon\eta]-\mathcal{L}[y]}{\varepsilon}=\frac{\partial}{\partial\varepsilon}\mathcal{L}[y+\varepsilon\eta]|_{\varepsilon=0}
\end{equation*}
(provided it exists). If the limit exists for all $\eta\in
C^{1}_{rd}$, then $\mathcal{L}$ is said to be \emph{G\^{a}teaux
differentiable} at $y$.
\end{df}
 The existence of G\^{a}teaux derivative $\delta\mathcal{L}[y;\eta]$
 presupposes that:\\
 (i) $\mathcal{L}[y]$ is defined;\\
 (ii) $\mathcal{L}[y+\varepsilon\eta]$ is defined for all sufficiently small
 $\varepsilon$.

\begin{Th}\label{lm}
Let $\mathcal{L},\mathcal{G}_{1},\ldots,\mathcal{G}_{m}$ be functionals
defined in a neighborhood of $\hat{y}$ and having continuous
G\^{a}teaux derivative in this neighborhood. Suppose that $\hat{y}$
is a weak local minimum of \eqref{vp} subject to the boundary
conditions \eqref{bc} and the isoperimetric constrains \eqref{con}.
Then, either: \\
(i) $ \forall v_{j}\in C^{1}_{rd}$, $j=1,\ldots,m$
\begin{equation}\label{lmi}
\left|
  \begin{array}{cccc}
    \delta\mathcal{G}_{1}[\hat{y};v_{1}] & \delta\mathcal{G}_{1}[\hat{y};v_{2}]& \cdots & \delta\mathcal{G}_{1}[\hat{y};v_{m}]\\
    \delta\mathcal{G}_{2}[\hat{y};v_{1}]& \delta\mathcal{G}_{2}[\hat{y};v_{2}] & \cdots & \delta\mathcal{G}_{2}[\hat{y};v_{m}] \\
    \vdots & \vdots& \vdots & \vdots \\
    \delta\mathcal{G}_{m}[\hat{y};v_{1}] & \delta\mathcal{G}_{m}[\hat{y};v_{2}] & \cdots & \delta\mathcal{G}_{m}[\hat{y};v_{m}] \\
  \end{array}
\right|=0
\end{equation}
or\\
(ii) there exist constants $\lambda_{i}\in \R$, $i=1,\ldots,m$ for
which
\begin{equation}\label{lmii}
\delta\mathcal{L}[\hat{y};\eta]=\sum_{i=1}^{m}\lambda_{i}\delta\mathcal{G}_{i}[\hat{y};\eta]
\quad \forall \eta \in C^{1}_{rd}.
\end{equation}
\end{Th}
\begin{proof}
This proof is patterned after the proof of Troutman \cite[Theorem
5.16]{Trout}. Let us consider, for fixed directions $\eta,
v_{1}, v_{2},\ldots,v_{m}$, the auxiliary functions:
\begin{equation*}
\begin{split}
l(p,q_{1},\ldots,q_{m})&= \mathcal{L}[\hat{y}+p\eta+q_{1}v_{1}+\cdots+q_{m}v_{m}],\\
g_{1}(p,q_{1},\ldots,q_{m})&= \mathcal{G}_{1}[\hat{y}+p\eta+q_{1}v_{1}+\cdots+q_{m}v_{m}],\\
&\vdots \\
g_{m}(p,q_{1},\ldots,q_{m})&=
\mathcal{G}_{m}[\hat{y}+p\eta+q_{1}v_{1}+\cdots+q_{m}v_{m}],
\end{split}
\end{equation*}
which are defined in some neighborhood of the origin in
$\R^{m+1}$, since
$\mathcal{L},\mathcal{G}_{1},\ldots,\mathcal{G}_{m}$ themselves are
defined in a neighborhood of $\hat{y}$. Note that the partial
derivative
\begin{equation*}
    l_{p}(p,q_{1},\ldots,q_{m})=\frac{\partial}{\partial
    p}l(p,q_{1},\ldots,q_{m})=\frac{\partial}{\partial
    p}\mathcal{L}[\hat{y}+p\eta+q_{1}v_{1}+\cdots+q_{m}v_{m}]
\end{equation*}
\begin{equation*}
=lim_{\varepsilon\rightarrow
    0}\frac{\mathcal{L}[\hat{y}+(p+\varepsilon)\eta+q_{1}v_{1}+\cdots
    +q_{m}v_{m}]-\mathcal{L}[\hat{y}+p\eta+q_{1}v_{1}+\cdots +q_{m}v_{m}]}{\varepsilon}
\end{equation*}
\begin{equation*}
=lim_{\varepsilon\rightarrow
    0}\frac{\mathcal{L}[y+\varepsilon\eta]-\mathcal{L}[y]}{\varepsilon},
\end{equation*}
with $y=\hat{y}+p\eta+q_{1}v_{1}+ \cdots +q_{m}v_{m}$. Therefore,
$l_{p}(p,q_{1},\ldots,q_{m})=\delta\mathcal{L}[\hat{y};\eta]$.
Similarly we have:
\begin{eqnarray*}
 l_{q_{i}}(p,q_{1},\ldots,q_{m})=\delta\mathcal{L}[\hat{y};v_{i}], \quad i=1,\ldots,m, \\
 (g_{j})_{p}(p,q_{1},\ldots,q_{m})=\delta\mathcal{G}_{j}[\hat{y};\eta], \quad j=1,\ldots,m,\\
(g_{j})_{q_{i}}(p,q_{1},\ldots,q_{m})=\delta\mathcal{G}_{j}[\hat{y};v_{i}],
\quad i=1,\ldots,m, \quad j=1,\ldots,m.
\end{eqnarray*}
Hence, the Jacobian determinant
$\frac{\partial(l,g_{1},\ldots,g_{m})}{\partial (p,q_{1},\ldots,q_{m})}$
evaluated at $(p,q_{1},\ldots,q_{m})=(0,0,\ldots,0)$ is the following:
\begin{equation}\label{jac}
\left|
  \begin{array}{cccc}
    \delta\mathcal{L}[\hat{y};\eta] & \delta\mathcal{L}_{1}[\hat{y};v_{1}]& \cdots & \delta\mathcal{L}_{1}[\hat{y};v_{m}]\\
    \delta\mathcal{G}_{1}[\hat{y};\eta]& \delta\mathcal{G}_{1}[\hat{y};v_{1}] & \cdots & \delta\mathcal{G}_{1}[\hat{y};v_{m}] \\
    \vdots & \vdots& \ddots & \vdots \\
    \delta\mathcal{G}_{m}[\hat{y};\eta] & \delta\mathcal{G}_{m}[\hat{y};v_{1}] & \cdots & \delta\mathcal{G}_{m}[\hat{y};v_{m}] \\
  \end{array}
\right|.
\end{equation}
Note also that the vector valued function $(l,g_{1},\ldots,g_{m})$
has continuous partial derivatives in a neighborhood of the
origin, since $\mathcal{L},\mathcal{G}_{1},\ldots,\mathcal{G}_{m}$
have continuous G\^{a}teaux derivative in the neighborhood of
$\hat{y}$. With this preparation we can prove our theorem. Assume
condition (i) does not hold for one set of directions: $v_{1},
v_{2},\ldots, v_{m}$ and suppose there exists one direction $\eta$
for which the determinant \eqref{jac} is nonvanishing. Therefore,
the classical inverse function theorem applies, i.e. the
application $(l,g_{1},\ldots,g_{m})$ maps a neighborhood of the
origin in $\R^{m+1}$ onto a region containing a full neighborhood
of
$(\mathcal{L}[\hat{y}],\mathcal{G}_{1}[\hat{y}],\ldots,\mathcal{G}_{m}[\hat{y}])$.
That is, one can find pre-image points
$(\acute{p},\acute{q}_{1},\ldots,\acute{q}_{m})$ and
$(\grave{p},\grave{q}_{1},\ldots,\grave{q}_{m})$ near the origin, for
which the points $\acute{y}=\hat{y}+\acute{p}\eta +
\Sigma_{i=1}^{m}\acute{q}_{i}v_{i}$ and
$\grave{y}=\hat{y}+\grave{p}\eta +
\Sigma_{i=1}^{m}\grave{q}_{i}v_{i}$ satisfy the conditions:
\begin{eqnarray*}
 \mathcal{L}[\acute{y}]<\mathcal{L}[\hat{y}]<\mathcal{L}[\grave{y}],\\
\mathcal{G}_{i}[\acute{y}]=\mathcal{G}_{i}[\hat{y}]=\mathcal{G}_{i}[\grave{y}],
\quad i=1,\ldots,m.
\end{eqnarray*}
This shows that $\hat{y}$ cannot be a local extremal for
$\mathcal{L}$ subject to constraints \eqref{con}, contradicting
the hypothesis. Thus, for the specific set of directions: $v_{1},
v_{2},\ldots, v_{m}$ the determinant \eqref{jac} must vanish for each
$\eta\in C^{1}_{rd}$. We expand it by minors of the first column
\begin{equation}\label{det}
\delta\mathcal{L}[\hat{y};\eta]\cdot cof
\delta\mathcal{L}[\hat{y};\eta]+\delta\mathcal{G}_{1}[\hat{y};\eta]\cdot
cof \delta\mathcal{G}_{1}[\hat{y};\eta]+
\ldots+\delta\mathcal{G}_{m}[\hat{y};\eta]\cdot cof
\delta\mathcal{G}_{m}[\hat{y};\eta]=0 \, ,
\end{equation}
where we are using the notation $cof$ to denote
the cofactor. Dividing equation \eqref{det} by $cof
\delta\mathcal{L}[\hat{y};\eta]$, since it is precisely the
nonvanishing determinant
 $\left|
 \begin{array}{c}
 \delta\mathcal{G}_{i}[\hat{y};v_{j}] \\
 i,j=1,\ldots,m\\
 \end{array}
 \right|$,
and setting
\begin{equation*}
\lambda_{i}=-\frac{cof \delta\mathcal{G}_{i}[\hat{y};\eta]}{cof
\delta\mathcal{L}[\hat{y};\eta]}
\end{equation*}
we obtain an equation equivalent to  \eqref{lmii}.
\end{proof}

Note that condition (ii) of Theorem \ref{lm} can be written in the
form
\begin{equation}\label{lmiii}
\delta \left(\mathcal{L}-\sum _{i=1}^{m}\lambda
_{i}\mathcal{G}_{i}[\hat{y};\eta]\right)=0 \quad \forall \eta \in
C^{1}_{rd},
\end{equation}
since the G\^{a}teaux derivative is a linear operation on the
functionals (by the linearity of the ordinary derivative).\\

Now, suppose that assumptions of Theorem \ref{lm} hold but
condition (i) does not hold. Then, equation \eqref{lmiii} is
fulfilled for every $\eta \in C^{1}_{rd}$. Let us consider function $\eta$
such that $\eta(a)=\eta(b)=0$ and denote by $\mathcal{F}$ the
functional $ \mathcal{L}-\sum _{i=1}^{m}\lambda
_{i}\mathcal{G}_{i}.$ Then we have
\begin{equation*}
0=\delta \mathcal{F}[\hat{y};\eta]=\frac{\partial}{\partial
\varepsilon}\mathcal{F}[\hat{y}+\varepsilon \eta]|_{\varepsilon=0}
\end{equation*}
\begin{equation*}
    =\int_{a}^{b}(F_{s}(t,\hat{y}^{\sigma}(t),\hat{y}^{\triangle}(t)
)\eta^{\sigma}(t)+F_{v}(t,\hat{y}^{\sigma}(t),\hat{y}^{\triangle}(t)
)\eta^{\triangle}(t))\triangle t,
\end{equation*}
where the function $F:[a,b]^{k}\times \R^{n}
\times \R^{n} \rightarrow \R$ is defined by $F(t,s,v)=L(t,s,v)-\sum _{i=1}^{m}\lambda _{i}G_{i}(t,s,v).$
Note that
\begin{equation*}
\int_{a}^{b}\left(\int_{a}^{t}F_{s}(\tau,\hat{y}^{\sigma}(\tau),\hat{y}^{\triangle}(\tau)
)\triangle \tau \eta(t)\right)^{{\triangle}}\triangle t=\int_{a}^{t}
F_{s}(\tau,\hat{y}^{\sigma}(\tau),\hat{y}^{\triangle}(\tau )
\triangle \tau \eta(t)|^{t=b}_{t=a}=0
\end{equation*}
and
\begin{multline*}
 \int_{a}^{b}\left(\int_{a}^{t}F_{s}(\tau,\hat{y}^{\sigma}(\tau),\hat{y}^{\triangle}(\tau)
)\triangle \tau \eta(t)\right)^{{\triangle}}\triangle t\\
= \int_{a}^{b}\left\{\left(\int_{a}^{t}
F_{s}\left(\tau,\hat{y}^{\sigma}(\tau),\hat{y}^{\triangle}(\tau)\right)
\triangle \tau\right)^{\triangle} \eta^{\sigma}(t) + \int_{a}^{t}
F_{s}(\tau,\hat{y}^{\sigma}(\tau),\hat{y}^{\triangle}(\tau)
)\triangle \tau \eta^{\triangle}(t)\right\}\triangle t
\end{multline*}
\begin{equation*}
=\int_{a}^{b}\left\{F_{s}(t,\hat{y}^{\sigma}(t),\hat{y}^{\triangle}(t)
)\eta^{\sigma}(t)+\int_{a}^{t}
F_{s}(\tau,\hat{y}^{\sigma}(\tau),\hat{y}^{\triangle}(\tau)
)\triangle \tau \eta^{\triangle}(t)\right\}\triangle t.
\end{equation*}
Therefore,
\begin{equation*}
0=\int_{a}^{b}\left\{F_{v}(t,\hat{y}^{\sigma}(t),\hat{y}^{\triangle}(t)
)-\int_{a}^{t}
F_{s}(\tau,\hat{y}^{\sigma}(\tau),\hat{y}^{\triangle}(\tau)
)\triangle \tau \right\}\eta^{\triangle}(t)\triangle t.
\end{equation*}
Since the function $\eta$ is arbitrary, Lemma~\ref{lemma:DR} implies
that
\begin{equation*}
F_{v}(t,\hat{y}^{\sigma}(t),\hat{y}^{\triangle}(t) )-\int_{a}^{t}
F_{s}(\tau,\hat{y}^{\sigma}(\tau),\hat{y}^{\triangle}(\tau)
)\triangle \tau=c
\end{equation*}
for some $c\in \R^{n}$ and all $t\in[a,b]^{k}$. Hence,
\begin{equation}\label{ele}
F_{v}^{\triangle}(t,\hat{y}^{\sigma}(t),\hat{y}^{\triangle}(t) )=
F_{s}(t,\hat{y}^{\sigma}(t),\hat{y}^{\triangle}(t))
\end{equation}
for all $t\in[a,b]^{k^{2}}$.

We have just proved the following  necessary optimality condition
for the isoperimetric problem with multiple constrains on time
scales.

\begin{Th}
\label{Th:B:EL-CV} Let us assumptions of Theorem \ref{lm} hold but
condition \eqref{lmi} does not hold.  If $\hat{y} \in
C_{\textrm{rd}}^1$ is a weak local minimum of the problem
\eqref{vp}-\eqref{con}, then it satisfies the Euler-Lagrange
equation \eqref{ele} for all $t\in[a,b]^{k^{2}}$.
\end{Th}


\subsection{Pareto optimality}
\label{subsec:po}

Let us consider a finite number $d\geq 1$ of (objective)
functionals:
\begin{equation}\label{mvp}
     \mathcal{L}_{i}[y]=\int_{a}^{b}L_{i}(t,y^{\sigma}(t),y^{\triangle}(t))\triangle
     t , \quad i=1,\ldots d ,
\end{equation}
where $a, b\in \mathbb {T}$ with $a<b$,
$L_{i}(t,s,v):[a,b]^{k}\times \R^{n} \times \R^{n} \rightarrow \R$,
$i=1,\ldots d$, have partial continuous derivatives with respect to
the second and third variables for all $t\in[a,b]^{k}$, and
$L_{i}(t,\cdot,\cdot)$ and theirs partial derivatives, $i=1,\ldots
d$, are rd-continuous at $t$. We would like to find a function $y\in
C^{1}_{rd}$, satisfying the boundary conditions \eqref{bc}, that
renders the minimum value to each functional $\mathcal{L}_i$, $i =
1,\ldots,d$, simultaneously. In general, there does not exist such a
function, and one uses the concept of Pareto optimality.
\begin{df}
A function $\hat{y}\in C^{1}_{rd}$ is called a \emph{weak locally
Pareto optimal solution} if there exists $\delta >0$ such that there
does not exist $y\in C^{1}_{rd}$ with
$\|y-\hat{y}\|_{C^{1}_{rd}}<\delta$ and
\begin{equation*}
\forall i\in\{1,\ldots,d\} :\mathcal{L}_{i}[y]\leqslant
\mathcal{L}_{i}[\hat{y}]\wedge \exists j \in\{1,\ldots,d\} :
\mathcal{L}_{j}[y] < \mathcal{L}_{j}[\hat{y}] \, .
\end{equation*}
\end{df}

\begin{Th}\label{sc}
If $\hat{y}$ is a weak local minimum of the functional
$\sum_{i=1}^{d}\gamma_{i}\mathcal{L}_{i}[y]$ with $\gamma_i > 0$
for $i = 1,\ldots,d$ and $\sum_{i=1}^{d} \gamma_i = 1$, then it is
a weak locally Pareto optimal solution of the multiobjective problem with functionals \eqref{mvp}.
\end{Th}
\begin{proof}
Let $\hat{y}$ be a weak local minimum of the functional
$\sum_{i=1}^{d}\gamma_{i}\mathcal{L}_{i}[y]$ with $\gamma_i > 0$
for $i = 1,\ldots,d$ and $\sum_{i=1}^{d} \gamma_i = 1$. Suppose on
the contrary that $\hat{y}$ is not a weak locally Pareto optimal. Then, for
every $\delta >0$ there exists $y$ with
$\|y-\hat{y}\|_{C^{1}_{rd}}<\delta$ such that $\forall
i\in\{1,\ldots,d\}$ we have $\mathcal{L}_{i}[y]\leqslant
\mathcal{L}_{i}[\hat{y}]$ and $\exists j \in\{1,\ldots,d\}$ such
that $\mathcal{L}_{j}[y] < \mathcal{L}_{j}[\hat{y}]$. Since
$\gamma_i
> 0$ for $i = 1,\ldots,d$, we obtain
$\sum_{i=1}^{d}\gamma_{i}\mathcal{L}_{i}[y]<
\sum_{i=1}^{d}\gamma_{i}\mathcal{L}_{i}[\hat{y}]$. This
contradicts our choice of $\hat{y}$.
\end{proof}

\begin{Th}\label{nc}
If  $\hat{y}$ is a weak locally Pareto optimal solution of the multiobjective problem with functionals
\eqref{mvp}, then it minimizes each one of the scalar functionals
\begin{equation*}
 \mathcal{L}_{i}[y]\, , \quad i \in \{1,\ldots,d\}
\end{equation*}
subject to the constraints
\begin{equation*}
\mathcal{L}_{j}[y]=\mathcal{L}_{j}[\hat{y}] \, , \quad j =
1,\ldots,d  \text{ and } j \ne i \, .
\end{equation*}
\end{Th}
\begin{proof}
Let $\hat{y}$ be a weak locally Pareto optimal solution of the
problem on time scales \eqref{mvp} and suppose the contrary,
\textrm{i.e.} that for some $i$ $\hat{y}$ does not solve the problem
$\mathcal{L}_{i}[y]\rightarrow min$ subject to
$\mathcal{L}_{j}[y]=\mathcal{L}_{j}[\hat{y}], j = 1,\ldots,d \, (j
\ne i).$ Then, for every $\delta >0$ there exists $y$ with
$\|y-\hat{y}\|_{C_{rd}^1} < \delta$ such that
$\mathcal{L}_{i}[y]<\mathcal{L}_{i}[\hat{y}]$ and
$\mathcal{L}_{j}[y]=\mathcal{L}_{j}[\hat{y}], j = 1,\ldots,d \, (j
\ne i)$. This contradicts the weak local Pareto optimality of
$\hat{y}$.
\end{proof}

\begin{Example}
Let $\mathbb {T}=\{0,1,2\}$. We would like to find locally Pareto
optimal solutions for
\begin{eqnarray*}
 \mathcal{L}_{1}[y]=\int_{0}^{2}y^{2}(t+1)\triangle t,\\
 \mathcal{L}_{2}[y]=\int_{0}^{2}(y(t+1)-2)^{2}\triangle t
\end{eqnarray*}
satisfying the boundary conditions $y(0)=0$, $y(2)=0$. Note that
\begin{equation*}
 \mathcal{L}_{1}[y]=\sum_{t=0}^{1}y^{2}(t+1) \, , \quad
 \mathcal{L}_{2}[y]=\sum_{t=0}^{1}(y(t+1)-2)^{2} \, ,
\end{equation*}
and that the possible solutions are of the form
\begin{gather*}
y(t)=
\begin{cases}
0 & \text{ if } t=0 \\
 a & \text{ if } t=1\\
 0 & \text{ if } t=2 \, ,
\end{cases}
\end{gather*}
where $a\in \R$. On account of the above, we have
$\mathcal{L}_{1}[y(t)]=a^{2}$, and
$\mathcal{L}_{2}[y(t)]=4+(a-2)^{2}$.
Using Theorem \ref{sc} we obtain that locally Pareto optimal solutions
for functionals $\mathcal{L}_{1}$, $\mathcal{L}_{2}$ are
\begin{gather*}
y(t)=
\begin{cases}
0 & \text{ if } t=0 \\
 a & \text{ if } t=1 \, ,\\
 0 & \text{ if } t=2
\end{cases}
\quad a\in [0,2] \, .
\end{gather*}
\end{Example}




\begin{thebibliography}{99}

\bibitem{b1} C. D. Ahlbrandt, M. Bohner\ and\ J. Ridenhour,
Hamiltonian systems on time scales, J. Math. Anal. Appl.
{\bf 250} (2000), no.~2, 561--578.

\bibitem{b2} B. Aulbach\ and\ S. Hilger,
Linear dynamic processes with inhomogeneous time scale,
in {\it Nonlinear dynamics and quantum dynamical
systems (Gaussig, 1990)}, 9--20, Akademie Verlag, Berlin.

\bibitem{Zbig} Z. Bartosiewicz\ and\ D. F. M. Torres,
Noether's theorem on time scales, J. Math. Anal. Appl. (accepted)
\texttt{arXiv:0709.0400}

\bibitem{b7} M. Bohner, Calculus of variations on time scales,
Dynam. Systems Appl. {\bf 13} (2004), no.~3-4, 339--349.

\bibitem{livro} M. Bohner\ and\ A. Peterson,
{\it Dynamic equations on time scales},
Birkh\"auser Boston, Boston, MA, 2001.

\bibitem{b3} Y. Censor, Pareto optimality in
multiobjective problems, Appl. Math. Optim.
{\bf 4} (1977/78), no.~1, 41--59.

\bibitem{b4} R. A. C. Ferreira\ and\ D. F. M. Torres,
Higher-order calculus of variations on time
scales. Proc. Workshop on Mathematical
Control Theory and Finance, Lisbon, 10-14 April 2007, pp.~150--158.
To appear in Springer---Business/Economics and Statistics (accepted).
\texttt{arXiv:0706.3141}.

\bibitem{b6} D. T. Luc\ and\ S. Schaible,
Efficiency and generalized concavity,
J. Optim. Theory Appl. {\bf 94} (1997), no.~1, 147--153.

\bibitem{b5} K. Miettinen, {\it Nonlinear multiobjective optimization},
Kluwer Acad. Publ., Boston, MA, 1999.

\bibitem{Trout} J. L. Troutman, {\it Variational calculus and optimal control},
Second edition, Springer, New York, 1996.

\end{thebibliography}
\end{document}